\documentclass{amsart}
\usepackage{amsfonts}
\usepackage{amsfonts,amssymb,amsmath,amsthm}
\usepackage{url}
\usepackage{enumerate}\usepackage{geometry}

\newtheorem{thm}{Theorem}[section]
\newtheorem{cor}[thm]{Corollary}

\newtheorem{lem}[thm]{Lemma}

\newtheorem{rem}[thm]{Remark}

\newcommand{\be}{\begin{equation}}
\newcommand{\ee}{\end{equation}}
\newcommand{\ben}{\begin{enumerate}}
\newcommand{\een}{\end{enumerate}}
\newcommand{\beq}{\begin{eqnarray}}
\newcommand{\eeq}{\end{eqnarray}}
\newcommand{\beqn}{\begin{eqnarray*}}
\newcommand{\eeqn}{\end{eqnarray*}}
\newcommand{\e}{\varepsilon}
\newcommand{\pa}{\partial}

\newcommand{\BK}{{\rm \bf K}}
\newcommand{\ms}{{\mathfrak{s}}}

\newcommand{\tF}{\tilde{F}}

\newcommand{\tx}{\tilde{x}}

\newcommand{\pzri}{{\pa \over \pa z^\gamma_i}}
\newcommand{\pzej}{{\pa \over \pa z^\eta_j}}

\newcommand{\tW}{{\tilde{W}}}

\newcommand{\tX}{{\tilde{X}}}

\newcommand{\la}{{\langle}}
\newcommand{\ra}{{\rangle}}

\newcommand{\tth}{\tilde{h}}

\begin{document}

\title[Nontrivial minimal surfaces in a hyperbolic Randers space]{\bf  Nontrivial minimal surfaces in a hyperbolic Randers space}

\author{Ningwei Cui$^1$,  Yi-Bing Shen$^2$ }
 \footnotetext{1. Supported by NSFC (No. 11401490) and the Fundamental Research Funds for the Central Universities (No. 2682014CX051) in China.
 
 2. Supported by NSFC (No. 11471246).}
\date{}

\address{address:
School of Mathematics, Southwest Jiaotong
 University, Chengdu, 610031, P.R.China.}
 \email{ningweicui@gmail.com}

\address{address: Department of Mathematics\\
Zhejiang University\\
Hangzhou, 310027\\
P.R.China}
 \email{yibingshen$@$zju.edu.cn}

\maketitle
\begin{abstract}
{The contribution of this paper is two-fold. The first one is to derive a simple formula of the mean curvature form for a hypersurface in the Randers space with a Killing  field, by considering the Busemann-Hausdorff measure and Holmes-Thompson measure simultaneously. The second one is to obtain the explicit local expressions of two types of nontrivial rotational  BH-minimal surfaces in a Randers domain of constant flag curvature $\BK=-1$, which are the first examples of BH-minimal surfaces in the hyperbolic Randers space. }
\end{abstract}

\section{Introduction}

The Randers manifold plays a fundamental role in the Finsler geometry.
The simply connected Randers manifolds of constant flag curvature are called the {\it  Randers space forms}, which were classified by using the Zermelo's navigation method in \cite{BRS}.  
The minimal surfaces in the  flat Randers space forms (\cite{CuiShen1, CuiShen2, RST, SS, ST,Wu}) and the positively curved Randers space forms (\cite{Cui1,Cui2})  have been studied in recent years.  
It is a common sense  that finding any explicit minimal surface is an interesting work  in Riemannian geometry. However, due to the complexity of the Finsler geometry, there is no known minimal surfaces in any negatively curved Finsler space up to the authors' best knowledge.
The negatively curved Randers space forms will be called {\em hyperbolic Randers space forms}  in this paper, and we {\em successfully} find explicit nontrivial minimal surfaces under the Busemann-Hausdorff measure in the hyperbolic Randers space form.

The author in \cite{Cui1} gave a formula of the mean curvature form in the Randers space by using the navigation data $(\tth,\tW)$ in the case that  $\tW$ is a Killing  field of constant length. However, according to the classification of  Randers space forms in \cite{BRS}, the hyperbolic Randers space forms  are modeled on the nonpositvely curved Riemannian space forms with conformal (Killing) vector fields.
In view of this, the formula in \cite{Cui1} in general
can not be used to study the hypersurfaces in the hyperbolic Randers space forms, because of the nonexistence  of nontrivial Killing fields of constant length in the Riemannian manifold of negative Ricci curvature.

By a straightforward computation, this paper will show that even in the case that $\tW$ is Killing only, the mean curvature formulas in the Randers space under both the Busemann-Hausdorff measure and Holmes-Thompson measure are extremely  simple,
 which generalize the formulas in \cite{Cui1}. Also surprisingly, under  the Busemann-Hausdorff measure, the formula enables us  to obtain the explicit local expressions of two types of nontrivial BH-minimal surfaces (spherical and hyperbolic types) 
 in an open domain of a hyperbolic Randers space form with flag curvature $\BK=-1$.

This paper is organized as follows. In Section 2, we give the formula of mean curvature form of a submanifold immersed in a Randers manifold by using Zermelo's navigation method (Theorem \ref{Killingmean}). This formula is  applicable for  both the Busemann-Hausdorff measure and the Holmes-Thompson measure, which has a very simple form when considering the hypersurface case and $\tW$ is Killing (Theorem \ref{geomean}). In Section 3, we consider an open domain of a hyperbolic Randers space form ($\BK=-1$) with a Killing field of mixed $S$ and $J$ type (\cite{BRS}), and then we obtain the explicit local expressions of nontrivial BH-minimal surfaces of spherical and hyperbolic types in  this domain (Theorem \ref{BHM}).

\section{Mean curvature of submanifolds by the method of navigation}

In this section, we derive the formula of mean curvature for an $n$-dimensional
submanifold isometrically immersed in an $(n+p)$-dimensional Randers manifold
$(\tilde{M},\tilde{F})$ with \emph{navigation data} $(\tth,\tW)$, where  $(\tth,\tW)$
consists of a  Riemannian metric 
$\tth=\sqrt{\tth_{\alpha\beta}d\tilde{x}^\alpha d\tilde{x}^\beta}$
and a vector field
 $\tW=\tW^\alpha{\partial\over\partial
\tilde{x}^\alpha}$  satisfying
$\|\tW_{\tilde{x}}\|_{\tth}<1$ at any point $\tilde{x}\in
\tilde{M}$. For the theory of Zermelo's navigation representation of Randers metrics, we refer to \cite{BRS}. In this paper, we
shall use the following convention of index ranges:
\[1\leq i,j,\cdot\cdot\cdot\leq n; \ \ \ 1\leq
\alpha,\beta,\cdot\cdot\cdot\leq n+p.\] Einstein summation
convention is also used throughout this paper.

 For an isometric immersion  $f:M^n\rightarrow
(\tilde{M}^{n+p},\tilde{F}=\tilde{\alpha}+\tilde{\beta})$,  the
induced metric $F=\alpha+\beta:=f^*\tilde{\alpha}+f^*\tilde{\beta}$
on $M$ is also a Randers metric. In local coordinates,  $f:M^n\rightarrow
(\tilde{M}^{n+p},\tilde{F})$ can be written as
\be\tilde{x}^{\alpha}=f^{\alpha}(x^1,\cdot\cdot\cdot,x^n).\label{graph}\ee

For a Finsler metric, there are two classical induced volume forms called Busemann-Hausdorff volume form and Holmes-Thompson volume form. At each point $x\in M$,  the volume form $dV_F$ of the induced
metric $F$ with respect to these two  volume forms can be written in an
uniform way \be {dV_F\mid}_x=\mathcal
{F}(\tilde{x},z)dx^1\wedge\cdot\cdot\cdot\wedge dx^n,\label{VVV}\ee
where  $\tilde{x}=f(x)\in \tilde{M}$ and $z=(z^\alpha_i)=({\pa
f^\alpha\over \pa x^i})$. The mean curvature form $\mathcal {H}_f=\mathcal {H}_\gamma d\tilde{x}^\gamma$ comes from the variation of volume functional of the
submanifolds with induced metrics (\cite{Shen}), precisely,
\be \mathcal {H}_\gamma={1\over \mathcal
{F}}\Big\{{ {\pa \mathcal {F} \over \pa \tilde{x}^\gamma}}-{ {\pa^2
\mathcal {F} \over \pa z^\gamma_i \pa z^\eta_j }}
  { {\pa^2 f^\eta \over \pa x^i \pa x^j
}}-{ {\pa^2 \mathcal {F}\over \pa \tilde{x}^\eta \pa z^\gamma_i }} {
{\pa f^\eta \over \pa x^i }}\Big\},\label{mean11}\ee
where $\mathcal {F}$ depends on the
volume forms (\cite{Shen}, \cite{Wu}).
Denote\be\ms:=1-\|\tilde{W}\|_{\tilde{h}}^2+\|W\|_{h}^2,\label{ms}\ee
where $\|W\|_{h}^2=h^{ij}W_iW_j$, $W_i=\tW_\alpha z^\alpha_i$, $(h^{ij})=(h_{ij})^{-1}$ and $h_{ij}=\tth_{\alpha\beta}z^\alpha_i z^\beta_j$. 
For the induced Randers metric $F$, the function $\mathcal{F}$ 
 can be
simultaneously given by 
\[\mathcal {F}={\varrho\over\chi}\sqrt{det(h_{ij})},\] where
$\varrho=\varrho(\ms)$ and $\chi$ are given by \beq
(\varrho,\chi)=\begin{cases}
(\ms^{-{n\over 2}},1), \ \ \ \ \  \ \ \mathrm{for\  the\  BH \ case}, \\
\Big(\ms^{1\over
2},(1-\|\tilde{W}\|_{\tilde{h}}^2)^{{n+1\over2}}\Big),\ \
\mathrm{for\ the\ HT \ case}.
\end{cases}\label{volume}\eeq
See Sec. 2 in \cite{Cui1}  for the details. The \emph{volume ratio function} for the Randers metric introduced in \cite{Cui1} is given by
\beq \Phi(\ms)&:=&2\varrho'(\ms)(1-\ms)+\varrho(\ms)\nonumber\\
&=&\begin{cases}
\ms^{-{n\over 2}}(-n\ms^{-{1}}+n+1), \ \ \ \ \  \ \ \mathrm{for\  the \ BH \ case},\\
\ms^{-{1\over
2}},\ \  \ \ \ \ \  \ \ \ \  \ \ \ \ \  \ \ \ \  \ \ \ \ \  \ \ \ \
\mathrm{for\ the\ HT \ case},
\end{cases} \label{Phi}\eeq
which appears naturally when we study the hypersurfaces isometrically immersed in Randers manifolds, and a similar function was given  in \cite{CuiShen1} for general $(\alpha,\beta)$-manifolds.

Let $\nabla^{\tilde{M}}$ be the Levi-Civita connection of
$(\tilde{M},\tth)$, 
$\tilde{W}_{\alpha|\eta}=\tth_{\alpha\beta}\tilde{W}^\beta_{|\eta}$
and let $\tilde{W}^\beta_{|\eta}$ be the coefficients of the covariant $(1,1)$-tensor $\nabla^{\tilde{M}}\tW$.
 Denote \be
A^{i\alpha}:=h^{ij}z^\alpha_j,\ \ \ \ \
 A^i_\alpha:=A^{i\delta}\tilde{h}_{\delta\alpha}, \ \ \ \ \ B^{\alpha\beta}:=h^{ij}z^\alpha_iz^\beta_j,\ \ \
B^\alpha_\beta:=B^{\alpha\delta}\tilde{h}_{\delta\beta}.
 \label{t_eq1}\ee

\begin{thm} \label{Killingmean} Let  $f:M^n\rightarrow
(\tilde{M}^{n+p},\tilde{F})$ be a submanifold isometrically
immersed in a Randers manifold, locally given by (\ref{graph}).
Suppose the navigation data of $(\tilde{M},\tilde{F})$ is
$(\tth,\tW)$. Then
 the
BH-mean curvature form (resp. the HT-mean curvature form) $\mathcal
{H}_f=\mathcal {H}_\gamma d\tilde{x}^\gamma$ of $f$ is given by\beq
\mathcal {H}_\gamma &=&{2(\mathrm{log}\chi)'\over
\varrho}\Big[2\varrho'
B^{\eta\delta}\tilde{W}_\delta\tilde{W}_\tau(\delta^\tau_\gamma-B^\tau_\gamma)-\varrho(\delta^\eta_\gamma-B^\eta_\gamma)\Big]
\tilde{W}_{\alpha|\eta}\tilde{W}^\alpha\nonumber\\
&&-{1\over
\varrho}\Big\{2\varrho'\Big[(\tilde{h}^{\alpha\beta}-B^{\alpha\beta})(\delta^\eta_\gamma-B^\eta_\gamma)+
B^{\eta\alpha}(\delta^\beta_\gamma-B^\beta_\gamma)\Big]\nonumber\\
&&-2B^{\eta\beta}\Big[2\varrho''(\tilde{h}^{\alpha\delta}-B^{\alpha\delta})\tilde{W}_\delta\tilde{W}_\tau(\delta^\tau_\gamma-B^\tau_\gamma)-\varrho'
(\delta^\alpha_\gamma-B^\alpha_\gamma)\Big]\Big\}\tilde{W}_{\alpha|\eta}\tilde{W}_\beta\nonumber\\
&&-{1\over
\varrho}\Big\{h^{ij}\Big[2\varrho'\tilde{W}_\alpha\tilde{W}_\beta(\delta^\alpha_\eta-B^\alpha_\eta)(\delta^\beta_\gamma-B^\beta_\gamma)+\varrho(\tilde{h}_{\gamma\eta}-B_{\gamma\eta})\Big]\nonumber\\
&&+2\tilde{W}_\delta\tilde{W}_\tau
A^{i\delta}A^{j\tau}\Big[2\varrho''\tilde{W}_\alpha\tilde{W}_\beta(\delta^\alpha_\eta-B^\alpha_\eta)(\delta^\beta_\gamma-B^\beta_\gamma)\nonumber\\
&&-\varrho'(\tilde{h}_{\gamma\eta}-B_{\gamma\eta})\Big]\Big\}\tau^\alpha_{ij},\label{mean}\eeq
where the
notation $\tau:=\tau^\alpha_{ij}dx^i\otimes
dx^j\otimes{\pa\over\pa{\tilde{x}^\alpha}}$ denotes the second
fundamental form of the immersion $f:(M,h)\rightarrow
(\tilde{M},\tilde{h})$, $\chi$ and $\varrho=\varrho(\ms)$ are given by (\ref{volume})
respectively, where we view $\chi$ as a function of $\|\tW\|^2_{\tth}$ and $(\mathrm{log}\chi)'$ is the derivative with respect to $\|\tW\|^2_{\tth}$.
\end{thm}

\begin{proof}  
The procedure is to compute the terms in (\ref{mean11}) one by one. 
Since $\|\tilde{W}\|^2_{\tilde{h}}$ is not necessarily of constant length, we compute \be{\pa \over \pa
\tilde{x}^\gamma}\|\tilde{W}\|^2_{\tilde{h}}=2\tilde{h}^{\alpha\beta}\tilde{W}_{\alpha|\gamma}\tilde{W}_\beta.\label{XXX}\ee

From
$h_{ij}=\tilde{h}_{\alpha\beta}z^\alpha_i z^\beta_j$ and
(\ref{t_eq1}), we compute \be{\pa \over \pa
\tilde{x}^\gamma}h^{ij}=-A^{i\tau}A^{j\eta}(\tth_{\delta\eta}
\tilde{\Gamma}^\delta_{\tau\gamma}+\tth_{\tau\delta}\tilde{\Gamma}^\delta_{\eta\gamma}),\label{AAA}\ee
where  $\tilde{\Gamma}^\alpha_{\tau\delta}$ denote the Christoffel
symbols of the Levi-Civita connection of $\tth$. From (\ref{t_eq1})
and (\ref{AAA}), we get \be{\pa  \over \pa
\tilde{x}^\gamma}B^{\alpha\beta}=-B^{\alpha\tau}B^{\beta\eta}(\tth_{\delta\eta}\tilde{\Gamma}^\delta_{\tau\gamma}+\tth_{\tau\delta}\tilde{\Gamma}^\delta_{\eta\gamma}).\label{BBB}\ee
By the definition of $\|W\|^2_{h}$ in (\ref{ms}), we have
$\|W\|^2_{h}=B^{\alpha\beta}\tW_\alpha\tW_\beta$. It follows from
(\ref{BBB}) that
 \beq {\pa \over \pa
\tilde{x}^\gamma}\|W\|^2_{h}&=&2B^{\alpha\beta}{\pa \over \pa
\tilde{x}^\gamma}\tilde{W}_{\alpha}\tilde{W}_\beta-2B^{\alpha\tau}
B^\beta_\delta\tilde{\Gamma}^\delta_{\tau\gamma}\tilde{W}_\alpha\tilde{W}_\beta\nonumber\\
&=&2B^{\alpha\beta}\tilde{W}_{\alpha|\gamma}\tilde{W}_\beta+2B^{\alpha\tau}(\delta^\beta_\delta-
B^\beta_\delta)\tilde{\Gamma}^\delta_{\tau\gamma}\tilde{W}_\alpha\tilde{W}_\beta.\label{Wh}\eeq

It is easy to show \be{\pa \over \pa
\tilde{x}^\gamma}\sqrt{det(h_{ij})}=B^\eta_\delta\tilde{\Gamma}^\delta_{\eta\gamma}\sqrt{det(h_{ij})}.\label{CXCX}\ee Let
\[\Pi^i_\delta:=\varrho
A^i_\delta+2\varrho'A^{i\alpha}(\delta^\beta_\delta-
B^\beta_\delta)\tilde{W}_\alpha\tilde{W}_\beta.\] Note that
$\mathcal {F}={\varrho\over\chi}\sqrt{det(h_{ij})}$ with
$(\varrho,\chi)$ given by (\ref{volume}).  From the equations (\ref{ms}), (\ref{XXX}),  (\ref{Wh}) and (\ref{CXCX}), we compute 
 \beq{ {\pa  \over \pa
\tilde{x}^\gamma}}\mathcal
{F}&=&2\Big[\varrho\Big({1\over\chi}\Big)'\tilde{h}^{\alpha\beta}-{\varrho'\over
\chi}(\tilde{h}^{\alpha\beta}-B^{\alpha\beta})\Big]\tilde{W}_{\alpha|\gamma}\tilde{W}_\beta\sqrt{det(h_{ij})}\nonumber\\
&&+{1\over \chi}\Big[ \varrho
B^\tau_\delta+2\varrho'B^{\alpha\tau}(\delta^\beta_\delta-
B^\beta_\delta)\tilde{W}_\alpha\tilde{W}_\beta\Big]\tilde{\Gamma}^\delta_{\tau\gamma}\sqrt{det(h_{ij})}\nonumber\\
&=&2\Big[\varrho\Big({1\over\chi}\Big)'\tilde{h}^{\alpha\beta}-{\varrho'\over
\chi}(\tilde{h}^{\alpha\beta}-B^{\alpha\beta})\Big]\tilde{W}_{\alpha|\gamma}\tilde{W}_\beta\sqrt{det(h_{ij})}\nonumber\\
&&+{1\over \chi}\Pi^i_\delta
z^\tau_i\tilde{\Gamma}^\delta_{\tau\gamma}\sqrt{det(h_{ij})}.\label{FX}
\eeq

Next, we compute the second term of (\ref{mean11}). Since $\|\tW\|_{\tth}$ is independent of $z^\alpha_i$, by
(\ref{t_eq1}), $h_{ij}=\tth_{\alpha\beta}z^\alpha_iz^\beta_j$ and
$\|W\|^2_{h}=B^{\alpha\beta}\tW_\alpha\tW_\beta$,
 we can get immediately
 \be\pzri
h^{kl}=-(h^{ki}A^l_\gamma+h^{li}A^k_\gamma),\label{hkl}\ee \be\pzri
B^{\alpha\beta}=A^{i\alpha}(\delta^\beta_\gamma-B^\beta_\gamma)+A^{i\beta}(\delta^\alpha_\gamma-B^\alpha_\gamma),\label{Balbt1}\ee
\be\pzri
\|W\|^2_{h}=2A^{i\alpha}(\delta^\beta_\gamma-B^\beta_\gamma)\tilde{W}_\alpha\tilde{W}_\beta.\label{Balbt}\ee

It is easy to show
 \be \pzri\sqrt{de
t(h_{ij})}=\sqrt{det(h_{ij})}A^i_\gamma.\label{DETA}\ee
 Since
$\mathcal {F}={\varrho\over\chi}\sqrt{det(h_{ij})}$, from (\ref{ms}),
(\ref{Balbt}) and (\ref{DETA}) we get
 \beq
\pzri \mathcal
{F}&=&{1\over\chi}\Big\{2\tilde{W}_\alpha\tilde{W}_\beta\varrho'
A^{i\alpha}(\delta^\beta_\gamma-B^\beta_\gamma)+\varrho
A^i_\gamma\Big\}\sqrt{de t(h_{ij})}\nonumber\\
&=:&{1\over\chi}\Pi^i_\gamma\sqrt{de t(h_{ij})}.\label{F1}\eeq By
(\ref{t_eq1}) and (\ref{hkl}), we can compute easily \be \pzej
A^{i\alpha}=h^{ij}(\delta^\alpha_\eta-B^\alpha_\eta)-A^{j\alpha}A^i_\eta.\label{A_ia1}\ee
Differentiating (\ref{F1}), using (\ref{Balbt1}), (\ref{Balbt}),
(\ref{DETA}), (\ref{A_ia1}) and by a direct computation, we get
 \beq {{\pa^2 \mathcal {F}\over \pa z^\gamma_i
\pa z^\eta_j }}&=&{1\over\chi}{{\pa\over \pa z^\eta_j
}}\Big(\Pi^i_\gamma\sqrt{de t(h_{ij})}\Big)\nonumber\\
&=&{1\over\chi}\Big\{\varrho(A^i_\gamma A^j_\eta-A^j_\gamma A^i_\eta)\nonumber\\
 &&+2\tilde{W}_\alpha\tilde{W}_\beta\varrho'(\delta^\beta_\gamma-B^\beta_\gamma)(A^{i\alpha}A^j_\eta-A^{j\alpha}A^i_\eta)\nonumber\\
 &&+2\tilde{W}_\alpha\tilde{W}_\beta\varrho'(\delta^\beta_\eta-B^\beta_\eta)(A^{j\alpha}A^i_\gamma-A^{i\alpha}A^j_\gamma)\nonumber\\
&&+h^{ij}\Big[2\tilde{W}_\alpha\tilde{W}_\beta\varrho'(\delta^\alpha_\eta-B^\alpha_\eta)(\delta^\beta_\gamma-B^\beta_\gamma)\nonumber\\
&&+\varrho(\tilde{h}_{\gamma\eta}-B_{\gamma\eta})\Big]\nonumber\\
&&+2\tilde{W}_\delta\tilde{W}_\tau
A^{i\delta}A^{j\tau}\Big[2\tilde{W}_\alpha\tilde{W}_\beta\varrho''(\delta^\alpha_\eta-B^\alpha_\eta)(\delta^\beta_\gamma-B^\beta_\gamma)\nonumber\\
&&-\varrho'(\tilde{h}_{\gamma\eta}-B_{\gamma\eta})\Big]\Big\}\sqrt{det(h_{ij})},\label{FZZ1}\eeq
where $B_{\gamma\eta}:=\tth_{\gamma\alpha}B^{\alpha}_\eta$.
  Contracting the equation (\ref{FZZ1}) gives
\beq && {{\pa^2 \mathcal {F}\over \pa z^\gamma_i
\pa z^\eta_j }}\Big[ {
{\pa^2 f^\eta \over \pa x^i \pa x^j
}}+\tilde{\Gamma}^\eta_{\tau\delta}z^\tau_jz^\delta_i\Big]\nonumber\\
&=&{1\over\chi}\Big\{h^{ij}\Big[2\tilde{W}_\alpha\tilde{W}_\beta\varrho'(\delta^\alpha_\eta-B^\alpha_\eta)(\delta^\beta_\gamma-B^\beta_\gamma)+\varrho(\tilde{h}_{\gamma\eta}-B_{\gamma\eta})\Big]\nonumber\\
&&+2\tilde{W}_\delta\tilde{W}_\tau
A^{i\delta}A^{j\tau}\Big[2\tilde{W}_\alpha\tilde{W}_\beta\varrho''(\delta^\alpha_\eta-B^\alpha_\eta)(\delta^\beta_\gamma-B^\beta_\gamma)\nonumber\\
&&-\varrho'(\tilde{h}_{\gamma\eta}-B_{\gamma\eta})\Big]\Big\}\Big[ {
{\pa^2 f^\eta \over \pa x^i \pa x^j
}}+\tilde{\Gamma}^\eta_{\tau\delta}z^\tau_jz^\delta_i\Big]\sqrt{det(h_{ij})}.\label{FZZ}\eeq

We then compute the third term of (\ref{mean11}). Note that
$\tW_{\alpha|\eta}$ and $\tilde{\Gamma}^\delta_{\tau\eta}$ are independent of $z^\gamma_i$. 
By
(\ref{FX}),  (\ref{Balbt1}), (\ref{Balbt}), (\ref{DETA})  and the first equality of
the equation (\ref{FZZ1}), we get immediately \beq { {\pa^2 \mathcal {F} \over \pa
\tilde{x}^\eta \pa z^\gamma_i }}&=&{\pa\over\pa z^\gamma_i
 }\Big({\pa\mathcal {F} \over \pa \tilde{x}^\eta}\Big)\nonumber\\
&=&\Big\{4\Big[\varrho'\Big({1\over\chi}\Big)'\tilde{h}^{\alpha\beta}-{\varrho''\over
\chi}(\tilde{h}^{\alpha\beta}-B^{\alpha\beta})\Big]A^{i\delta}\tilde{W}_\delta\tilde{W}_\tau(\delta^\tau_\gamma-B^\tau_\gamma)\nonumber\\
&&+2\Big[\varrho\Big({1\over\chi}\Big)'\tilde{h}^{\alpha\beta}-{\varrho'\over
\chi}(\tilde{h}^{\alpha\beta}-B^{\alpha\beta})\Big]A^i_\gamma \nonumber\\
&&+2{\varrho'\over
\chi}\Big[A^{i\alpha}(\delta^\beta_\gamma-B^\beta_\gamma)+A^{i\beta}(\delta^\alpha_\gamma-B^\alpha_\gamma)\Big]\Big\}
\tilde{W}_{\alpha|\eta}\tilde{W}_\beta\sqrt{det(h_{ij})}\nonumber\\
&&+{1\over \chi}\Big[\pzri\Big(\Pi^k_\delta \sqrt{det(h_{ij})}
\Big)z^\tau_k+\Pi^i_\delta \sqrt{det(h_{ij})}
\delta^\tau_\gamma\Big]\tilde{\Gamma}^\delta_{\tau\eta}\nonumber\eeq

\beq&=&\Big\{4\Big[\varrho'\Big({1\over\chi}\Big)'\tilde{h}^{\alpha\beta}-{\varrho''\over
\chi}(\tilde{h}^{\alpha\beta}-B^{\alpha\beta})\Big]A^{i\delta}\tilde{W}_\delta\tilde{W}_\tau(\delta^\tau_\gamma-B^\tau_\gamma)\nonumber\\
&&+2\Big[\varrho\Big({1\over\chi}\Big)'\tilde{h}^{\alpha\beta}-{\varrho'\over
\chi}(\tilde{h}^{\alpha\beta}-B^{\alpha\beta})\Big]A^i_\gamma \nonumber\\
&&+2{\varrho'\over
\chi}\Big[A^{i\alpha}(\delta^\beta_\gamma-B^\beta_\gamma)+A^{i\beta}(\delta^\alpha_\gamma-B^\alpha_\gamma)\Big]\Big\}
\tilde{W}_{\alpha|\eta}\tilde{W}_\beta\sqrt{det(h_{ij})}\nonumber\\
&&+{1\over \chi}\Big[\chi{ {\pa^2 \mathcal {F} \over \pa z^\delta_k \pa
z^\gamma_i }}z^\tau_k+\Pi^i_\delta
\delta^\tau_\gamma\sqrt{det(h_{ij})}\Big]\tilde{\Gamma}^\delta_{\tau\eta}.\label{FXZ}
\eeq

Now by (\ref{FX}), (\ref{FZZ}) and (\ref{FXZ}), we compute
the mean curvature form (\ref{mean11}): \beq \mathcal {H}_\gamma &=&-{1\over \mathcal {F}}\Big\{-{ {\pa
\mathcal {F} \over \pa \tilde{x}^\gamma}}+{ {\pa^2 \mathcal {F}
\over \pa z^\gamma_i \pa z^\eta_j }}
  { {\pa^2 f^\eta \over \pa x^i \pa x^j
}}+{ {\pa^2 \mathcal {F} \over \pa \tilde{x}^\eta \pa
z^\gamma_i }} { {\pa f^\eta \over \pa x^i }}\Big\}\nonumber\\
&=&-{1\over \mathcal {F}}\Big\{-2\Big[\varrho\Big({1\over\chi}\Big)'\tilde{h}^{\alpha\beta}-{\varrho'\over
\chi}(\tilde{h}^{\alpha\beta}-B^{\alpha\beta})\Big]\tilde{W}_{\alpha|\gamma}\tilde{W}_\beta\nonumber\\
&&-{1\over \chi}\Pi^i_\delta
z^\tau_i\tilde{\Gamma}^\delta_{\tau\gamma}
+{ {\pa^2 \mathcal {F}
\over \pa z^\gamma_i \pa z^\eta_j }}
  { {\pa^2 f^\eta \over \pa x^i \pa x^j
}}\nonumber\\
&&+\Big(4\Big[\varrho'\Big({1\over\chi}\Big)'\tilde{h}^{\alpha\beta}-{\varrho''\over
\chi}(\tilde{h}^{\alpha\beta}-B^{\alpha\beta})\Big]B^{\eta\delta}\tilde{W}_\delta\tilde{W}_\tau(\delta^\tau_\gamma-B^\tau_\gamma)\nonumber\\
&&+2\Big[\varrho\Big({1\over\chi}\Big)'\tilde{h}^{\alpha\beta}-{\varrho'\over
\chi}(\tilde{h}^{\alpha\beta}-B^{\alpha\beta})\Big]B^\eta_\gamma  \nonumber\\
&&+2{\varrho'\over
\chi}\Big[B^{\eta\alpha}(\delta^\beta_\gamma-B^\beta_\gamma)+B^{\eta\beta}(\delta^\alpha_\gamma-B^\alpha_\gamma)\Big]\Big)
\tilde{W}_{\alpha|\eta}\tilde{W}_\beta+{1\over \chi}\Pi^i_\delta
\tilde{\Gamma}^\delta_{\gamma\eta}z^\eta_i\Big\}\sqrt{det(h_{ij})}\nonumber\\
&&-{1\over \mathcal {F}}{ {\pa^2 \mathcal {F} \over \pa z^\delta_k \pa
z^\gamma_i }}\tilde{\Gamma}^\delta_{\tau\eta}z^\tau_kz^\eta_i\nonumber\\
&=&-{\chi\over
\varrho}\Big\{-2\Big[\varrho\Big({1\over\chi}\Big)'\tilde{h}^{\alpha\beta}-{\varrho'\over
\chi}(\tilde{h}^{\alpha\beta}-B^{\alpha\beta})\Big](\delta^\eta_\gamma-B^\eta_\gamma)\nonumber\\
&&+4\Big[\varrho'\Big({1\over\chi}\Big)'\tilde{h}^{\alpha\beta}-{\varrho''\over
\chi}(\tilde{h}^{\alpha\beta}-B^{\alpha\beta})\Big]B^{\eta\delta}\tilde{W}_\delta\tilde{W}_\tau(\delta^\tau_\gamma-B^\tau_\gamma)\nonumber\\
&&+2{\varrho'\over
\chi}\Big[B^{\eta\alpha}(\delta^\beta_\gamma-B^\beta_\gamma)+B^{\eta\beta}
(\delta^\alpha_\gamma-B^\alpha_\gamma)\Big]\Big\}
 \tilde{W}_{\alpha|\eta}\tilde{W}_\beta\nonumber\\
 &&-{1\over \mathcal {F}}{{\pa^2 \mathcal {F}\over \pa z^\gamma_i \pa z^\eta_j
}}\Big[ { {\pa^2 f^\eta \over \pa x^i \pa x^j
}}+\tilde{\Gamma}^\eta_{\tau\delta}z^\tau_jz^\delta_i\Big]\nonumber\eeq
\beq&=&-{2(\mathrm{log}\chi)'\over
\varrho}\Big[\varrho(\delta^\eta_\gamma-B^\eta_\gamma)-2\varrho'
B^{\eta\delta}\tilde{W}_\delta\tilde{W}_\tau(\delta^\tau_\gamma-B^\tau_\gamma)\Big]
\tilde{W}_{\alpha|\eta}\tilde{W}^\alpha\nonumber\\
&&-{1\over
\varrho}\Big\{2\varrho'\Big[(\tilde{h}^{\alpha\beta}-B^{\alpha\beta})(\delta^\eta_\gamma-B^\eta_\gamma)+
B^{\eta\alpha}(\delta^\beta_\gamma-B^\beta_\gamma)\Big]\nonumber\\
&&-2B^{\eta\beta}\Big[2\varrho''(\tilde{h}^{\alpha\delta}-B^{\alpha\delta})\tilde{W}_\delta\tilde{W}_\tau(\delta^\tau_\gamma-B^\tau_\gamma)-\varrho'
(\delta^\alpha_\gamma-B^\alpha_\gamma)\Big]\Big\}\tilde{W}_{\alpha|\eta}\tilde{W}_\beta\nonumber\\
&&-{1\over
\varrho}\Big\{h^{ij}\Big[2\varrho'\tilde{W}_\alpha\tilde{W}_\beta(\delta^\alpha_\eta-B^\alpha_\eta)(\delta^\beta_\gamma-B^\beta_\gamma)+\varrho(\tilde{h}_{\gamma\eta}-B_{\gamma\eta})\Big]\nonumber\\
&&+2\tilde{W}_\delta\tilde{W}_\tau
A^{i\delta}A^{j\tau}\Big[2\varrho''\tilde{W}_\alpha\tilde{W}_\beta(\delta^\alpha_\eta-B^\alpha_\eta)(\delta^\beta_\gamma-B^\beta_\gamma)\nonumber\\
&&-\varrho'(\tilde{h}_{\gamma\eta}-B_{\gamma\eta})\Big]\Big\}\Big[ {
{\pa^2 f^\eta \over \pa x^i \pa x^j
}}+\tilde{\Gamma}^\eta_{\tau\delta}z^\tau_jz^\delta_i\Big].\label{mean1}\eeq

It is well known that for a Riemannian isometric immersion
$f:(M,h)\rightarrow (\tilde{M},\tth)$, the second fundamental form
is given by $\tau:=\tau^\alpha_{ij}dx^i\otimes
dx^j\otimes{\pa\over\pa{\tilde{x}^\alpha}}$, where
\be \tau^\alpha_{ij}:=z^\alpha_{ij}
+\tilde{\Gamma}^\alpha_{\tau\delta}z^\tau_jz^\delta_i-\Gamma^k_{ij}z^\alpha_k,\label{tau_or}\ee
$z^\alpha_{ij}:={\pa^2 f^\alpha \over \pa x^i \pa x^j }$,
$\Gamma^k_{ij}$ and $\tilde{\Gamma}^\alpha_{\tau\delta}$ are the
Christoffel symbols of the corresponding  Levi-Civita connections of
$(M,h)$ and $(\tilde{M},\tth)$, respectively. Since $f$ is
isometric, $h_{ij}=\tth_{\alpha\beta}z^\alpha_iz^\beta_j$, a direct
computation gives
\[\Gamma^k_{ij}=A^k_\eta(z^\eta_{ij}
+\tilde{\Gamma}^\eta_{\tau\delta}z^\tau_jz^\delta_i).\] Then
\be\tau^\alpha_{ij}=(\delta^\alpha_\eta-B^\alpha_\eta)(z^\eta_{ij}
+\tilde{\Gamma}^\eta_{\tau\delta}z^\tau_jz^\delta_i).\label{tau}\ee
Therefore (\ref{mean}) follows from (\ref{mean1}) and
(\ref{tau}) immediately.\end{proof}

If $\tW$ is a Killing field,  then $\tW_{\alpha|\beta}+\tW_{\beta|\alpha}=0$ and Theorem \ref{Killingmean} reduces to the following theorem by a straightforward computation.

\begin{thm} \label{Killingmean2} Let  $f:M^n\rightarrow
(\tilde{M}^{n+p},\tilde{F})$ be a submanifold isometrically
immersed in a Randers manifold, locally given by (\ref{graph}).
Suppose the navigation data of $(\tilde{M},\tilde{F})$ is
$(\tth,\tW)$, where $\tilde{W}$ is a Killing  vector field. Then
 the
BH-mean curvature form (resp. the HT-mean curvature form) $\mathcal
{H}_f=\mathcal {H}_\gamma d\tilde{x}^\gamma$ of $f$ is given by\beq
\mathcal{H}_\gamma&=&{2(\mathrm{log}\chi)'\over
\varrho}\Big[2\varrho'
B^{\eta\delta}\tilde{W}_\delta\tilde{W}_\tau(\delta^\tau_\gamma-B^\tau_\gamma)-\varrho(\delta^\eta_\gamma-B^\eta_\gamma)\Big]
\tilde{W}_{\alpha|\eta}\tilde{W}^\alpha\nonumber\\
&&-{1\over \varrho}\Big\{4\varrho'B^{\eta\beta}
\tilde{W}_{\alpha|\eta}\tilde{W}_\beta(\delta^\alpha_\gamma-B^\alpha_\gamma)-2\Big[2\varrho''
B^{\eta\delta}\tilde{W}_\delta\tilde{W}_\tau(\delta^\tau_\gamma-B^\tau_\gamma)-\varrho'(\delta^\eta_\gamma-B^\eta_\gamma)\Big]
\tilde{W}_{\alpha|\eta}\tilde{W}^\alpha\Big\}\nonumber\\
&&-{1\over \varrho}\Big\{h^{ij}\Big[2\varrho'\tilde{W}_\alpha\tilde{W}_\beta(\delta^\beta_\gamma-B^\beta_\gamma)+\varrho\tilde{h}_{\gamma\alpha}\Big]\nonumber\\
&&+2\tilde{W}_\delta\tilde{W}_\tau
A^{i\delta}A^{j\tau}\Big[2\varrho''\tilde{W}_\alpha\tilde{W}_\beta(\delta^\beta_\gamma-B^\beta_\gamma)-\varrho'
\tilde{h}_{\gamma\alpha}\Big]\Big\}\tau^\alpha_{ij},\label{killingmean}\eeq
where the
notations are the same as those in Theorem \ref{Killingmean}.
\end{thm}

\begin{rem} If $\tW$ is a Killing field of constant length, then $\tilde{W}_{\alpha|\eta}\tilde{W}^\alpha=0$ and Theorem \ref{Killingmean2} reduces to Lemma 3.2 in \cite{Cui1}. Although the formula (\ref{killingmean}) is much more complicate than that in \cite{Cui1},
 it will be sharply simplified in the hypersurface case by using the following lemma, which is really surprising.
\end{rem}

Now we assume that  $f:(M^n,h)\rightarrow
(\tilde{M}^{n+1},\tilde{h})$
is a hypersurface and $N$ is a unit normal
vector field  on $f(M)$.   
  We shall use
$\pa_\alpha$ instead of ${\pa\over\pa\tilde{x}^\alpha}$ and
$\la,\ra$ instead of $\la,\ra_{\tth}$ for simplicity of notations.

\begin{lem} \label{MMM}   In the case of hypersurface $f:(M^n,h)\rightarrow
(\tilde{M}^{n+1},\tilde{h})$ with navigation data $(\tth,\tW)$, where $\tilde{W}$ is a Killing vector field,  for any $\tX\in T\tilde{M}$ along $f$, we have
  \be B^{\eta\beta}\tilde{W}_\beta \tilde{W}_{\alpha|\eta}\tW^\alpha=\la\nabla^{\tilde{M}}_{df(W)}\tilde{W}),\tW\ra, \ \ \ \ \  \tW_\tau(\delta^\tau_\gamma-B^\tau_\gamma)\tX^\gamma=w\la N,\tX \ra, \label{imp000}\ee
  \be B^{\eta\beta}\tilde{W}_\beta
\tilde{W}_{\alpha|\eta}(\delta^\alpha_\gamma-B^\alpha_\gamma)\tX^\gamma=\la \nabla^{\tilde{M}}_{df(W)}\tilde{W},N\ra \la N,\tX\ra, \ \ \ \ h^{ij}\tau^\alpha_{ij}\tW_\alpha=nHw,\label{impq}\ee

\be
 \tW_\delta\tW_\tau
A^{i\delta}A^{j\tau}\tau^\alpha_{ij}\tX_\alpha=-\Big[\la df(\nabla w),\tW\ra +{1\over 2}N(\|\tW\|_{\tth}^2)\Big]\la N,\tX \ra, \label{CCC1}\ee

\be
 \tW_\delta\tW_\tau
A^{i\delta}A^{j\tau}\tau^\alpha_{ij}\tW_\alpha=-\Big[\la df(\nabla w),\tW\ra +{1\over 2}N(\|\tW\|_{\tth}^2)\Big]w,  \label{CCC2}\ee
\be\la \nabla^{\tilde{M}}_{df(W)}\tilde{W},N\ra=-{1\over 2}N(\|\tW\|_{\tth}^2),\ \ \ \ \ \la \nabla^{\tilde{M}}_{df(W)}\tilde{W},\tW\ra=-{1\over 2}N(\|\tW\|_{\tth}^2)w, \label{CCC3} \ee
where $w=\la N,\tW\ra$ and $H$ is the mean curvature of $f$ with respect to $N$.
\end{lem}
\begin{proof}
First,  for any vector
field $\tilde{Y}=\tilde{Y}^\alpha\pa_\alpha\in T\tilde{M}$ along
$f$, we have \be\la \tilde{Y},N\ra
N=\tilde{Y}_\alpha(\tth^{\alpha\gamma}-B^{\alpha\gamma})\pa_\gamma,\label{Y}\ee
where $\tilde{Y}_\alpha=\tth_{\alpha\beta}\tilde{Y}^\beta$. (See the proof in Page 93 in \cite{Cui1}).

 For the index $1\leq\alpha\leq n+1$, recalling that  $W=W^j{\pa\over\pa x^j}$ and  $W^j=h^{ij}\tilde{W}_\beta
z^\beta_i$, and noticing the notation (\ref{t_eq1}), we compute \be
[\nabla^{\tilde{M}}_{df(W)}\tilde{W}]_\alpha=
[W^jz^\eta_j\nabla_{\pa_\eta}^{\tilde{M}}\tilde{W}]_\alpha=W^jz^\eta_j\tilde{W}_{\alpha|\eta}=
B^{\eta\beta}\tilde{W}_\beta\tilde{W}_{\alpha|\eta},\label{GG}\ee
which implies the first equation of (\ref{imp000}) immediately. Taking $\tilde{Y}=\tilde{W}$ in (\ref{Y}) and noticing $w=\la N,\tW\ra$, we prove the second equation of (\ref{imp000}). 
Taking $\tilde{Y}=\nabla^{\tilde{M}}_{df(W)}\tilde{W}$ in (\ref{Y}), we have
 \be
\la \nabla^{\tilde{M}}_{df(W)}\tilde{W},N\ra N=
[\nabla^{\tilde{M}}_{df(W)}\tilde{W}]_\alpha(\tth^{\alpha\gamma}-B^{\alpha\gamma})\pa_\gamma= 
B^{\eta\beta}\tilde{W}_\beta\tilde{W}_{\alpha|\eta}(\tth^{\alpha\gamma}-B^{\alpha\gamma})\pa_\gamma,\label{GG}\ee
which implies the first equation of  (\ref{impq}) immediately.   Note that the Riemannian mean curvature
$H$ for the hypersurface in the direction of $N$ is defined by
$h^{ij}\tau^\alpha_{ij}{\pa\over\pa\tilde{x}^\alpha}=nHN$. Noticing $w=\la N,\tW\ra$, we prove the second equation of  (\ref{impq}).

Taking $\tilde{Y}=\tilde{W}$ in (\ref{Y}), we get 
$\la \tilde{W},N\ra
N=\tilde{W}_\alpha(\tth^{\alpha\gamma}-B^{\alpha\gamma})\pa_\gamma=\tW-df(W)$, and then
$df(W)=\tW-w N$.  For $X,Y\in TM$, it follows from (\ref{tau_or}) that
\[\tau(X,Y)=\nabla^{\tilde{M}}_{df(X)}df(Y)-df(\nabla^M_XY),\]where  $\nabla^{M}$ is the Levi-Civita connection of
$(M,h)$.  For the hypersurface $M$ and the chosen unit normal field  $N$, we denote $\tau(X,Y)=:B(X,Y)N$. 
Then we compute
\beq \la df(\nabla w),\tW\ra&=&\la\nabla w,W\ra_h=W\la N,\tW\ra=\la \nabla^{\tilde{M}}_{df(W)} N,\tW\ra+\la
N,\nabla^{\tilde{M}}_{df(W)}\tW\ra\nonumber\\
&=&\la \nabla^{\tilde{M}}_{df(W)} N,\tW\ra+\la
N,\nabla^{\tilde{M}}_{\tW}\tW\ra-w\la
N,\nabla^{\tilde{M}}_{N}\tW\ra\nonumber\\
&=&\la \nabla^{\tilde{M}}_{df(W)} N,\tW\ra+\la
N,\nabla^{\tilde{M}}_{\tW}\tW\ra\nonumber\\
&=&\la \nabla^{\tilde{M}}_{df(W)} N,df(W)\ra-\la
\tW,\nabla^{\tilde{M}}_{N}\tW\ra=-B(W,W)-{1\over 2}N(\|\tW\|_{\tth}^2),\nonumber \eeq 
where  we use  $B(W,W)=\la \tau(W,W),N\ra=\la \nabla^{\tilde{M}}_{df(W)} df(W),N\ra$. Recalling that $W=h^{ij}\tilde{W}_\alpha
z^\alpha_i{\pa\over\pa x^j}= A^{j\alpha}\tilde{W}_\alpha{\pa\over\pa
x^j}\in TM$,  we get that (\ref{CCC1}) follows from \beq \tW_\delta\tW_\tau
A^{i\delta}A^{j\tau}\tau^\alpha_{ij}\tX_\alpha&=&W^{i}W^{j}\tau^\alpha_{ij}\tX_\alpha=\la
\tau(W,W),\tX\ra=B(W,W)\la N,\tX\ra\nonumber\\
&=&-\Big[\la df(\nabla w),\tW\ra+{1\over 2}N(\|\tW\|_{\tth}^2)\Big] \la
N,\tX\ra \label{CC1}\nonumber\eeq and (\ref{CCC2}) is proved by taking $\tX=\tW$ in (\ref{CCC1}).
Similarly, since $df(W)=\tW-w N$, (\ref{CCC3}) follows from \[\la \nabla^{\tilde{M}}_{df(W)}\tilde{W},N\ra=\la \nabla^{\tilde{M}}_{\tW}\tilde{W},N\ra-w\la \nabla^{\tilde{M}}_{N}\tilde{W},N\ra=-{1\over 2}N(\|\tW\|_{\tth}^2)\]
and
\[\la \nabla^{\tilde{M}}_{df(W)}\tilde{W},\tW\ra=\la \nabla^{\tilde{M}}_{\tW}\tilde{W},\tW\ra-w\la \nabla^{\tilde{M}}_{N}\tilde{W},\tW\ra=-{1\over 2}N(\|\tW\|_{\tth}^2)w.\]
\end{proof}

The formula (\ref{killingmean}) is very complicated, but luckily,  by using Lemma \ref{MMM}, several terms in (\ref{killingmean}) can be annihilated and the final form seems quite simple. We believe that this formula will play an important role in the future study of cmc surfaces in Randers space forms.

\begin{thm}{\label{geomean}}Let  $f:M^n\rightarrow (\tilde{M}^{n+1},\tilde{F})$ be a hypersurface
isometrically immersed in a Randers manifold. Suppose the navigation
data of $(\tilde{M},\tilde{F})$ is $(\tth,\tW)$, where
$\tilde{W}$ is a Killing vector field. Then for any $\tilde{X}\in T\tilde{M}$ along $f$,
 the
BH-mean curvature (resp. the HT-mean curvature) of $f$ is given by \[
\mathcal {H}_f(\tilde{X})=-{1\over \varrho}\Big\{\Big[nH+N(\mathrm{log}\chi)\Big]\Phi(\ms)-2\la
df(\nabla w), \tW\ra_{\tth}\Phi'(\ms)\Big\}\la
N,\tilde{X}\ra_{\tth},\] where $N$ is a unit normal field of
$f:(M,h)\rightarrow (\tilde{M},\tth)$, $H$ is the Riemannian mean
curvature with respect to $N$, $\nabla w$ is the gradient of $w$
with respect to $h$, $w:=\la N,\tW\ra_{\tth}$, $\ms=1-w^2$,  $\chi$ is given by (\ref{volume})  and
$\Phi(\ms)$ is given by (\ref{Phi}), respectively.
\end{thm}

Proof: By (\ref{ms}) and (\ref{Y}), we get
\[\ms=1-(\tth^{\alpha\beta}-B^{\alpha\beta})\tW_\alpha\tW_\beta
=1-\la N,\tW\ra^2=1-w^2.\]  Then by using Theorem \ref{Killingmean2} and Lemma \ref{MMM},  we compute 

\beq \mathcal
{H}_f(\tX)&=&\mathcal
{H}_\gamma \tX^\gamma\nonumber\\
&=&{2(\mathrm{log}\chi)'\over
\varrho}\Big[2\varrho'
B^{\eta\delta}\tilde{W}_\delta\tilde{W}_\tau(\delta^\tau_\gamma-B^\tau_\gamma)-\varrho(\delta^\eta_\gamma-B^\eta_\gamma)\Big]
\tilde{W}_{\alpha|\eta}\tilde{W}^\alpha\tX^\gamma\nonumber\\
&&-{1\over \varrho}\Big\{4\varrho'B^{\eta\beta}
\tilde{W}_{\alpha|\eta}\tilde{W}_\beta(\delta^\alpha_\gamma-B^\alpha_\gamma)-2\Big[2\varrho''
B^{\eta\delta}\tilde{W}_\delta\tilde{W}_\tau(\delta^\tau_\gamma-B^\tau_\gamma)-\varrho'(\delta^\eta_\gamma-B^\eta_\gamma)\Big]
\tilde{W}_{\alpha|\eta}\tilde{W}^\alpha\Big\}\tX^\gamma\nonumber\\
&&-{1\over \varrho}\Big\{h^{ij}\Big[2\varrho'\tilde{W}_\alpha\tilde{W}_\beta(\delta^\beta_\gamma-B^\beta_\gamma)+\varrho\tilde{h}_{\gamma\alpha}\Big]\nonumber\\
&&+2\tilde{W}_\delta\tilde{W}_\tau
A^{i\delta}A^{j\tau}\Big[2\varrho''\tilde{W}_\alpha\tilde{W}_\beta(\delta^\beta_\gamma-B^\beta_\gamma)-\varrho'
\tilde{h}_{\gamma\alpha}\Big]\Big\}\tau^\alpha_{ij}\tX^\gamma\nonumber\\
 &=&{2(\mathrm{log}\chi)'\over
\varrho}\Big[2\varrho'
\la \nabla^{\tilde{M}}_{df(W)}\tilde{W},\tW\ra w-{1\over 2}\varrho N(\|\tW\|_{\tth}^2)\Big]\la N,\tX\ra\nonumber\\
&&-{1\over \varrho}\Big\{4\varrho'\la \nabla^{\tilde{M}}_{df(W)}\tilde{W},N\ra -2\Big[2\varrho''\la \nabla^{\tilde{M}}_{df(W)}\tilde{W},\tW\ra w-{1\over 2}\varrho'N(\|\tW\|_{\tth}^2)\Big]\Big\} \la N,\tX\ra
\nonumber\\
&&-{1\over \varrho}\Big\{nH\Big[2\varrho'w^2+\varrho\Big]-\Big(2\la df(\nabla w), \tW\ra+N(\|\tW\|_{\tth}^2)\Big)\Big[2\varrho''w^2-\varrho'\Big]\Big\} \la N,\tX\ra\nonumber\\
&=&-{(\mathrm{log}\chi)'\over
\varrho}\Big[2\varrho'N(\|\tW\|_{\tth}^2)w^2+\varrho N(\|\tW\|_{\tth}^2)\Big]
 \la N,\tX\ra \nonumber\\
&&-{1\over \varrho}\Big\{-2\varrho'N(\|\tW\|_{\tth}^2) +\Big[2\varrho''N(\|\tW\|_{\tth}^2)w^2+\varrho'N(\|\tW\|_{\tth}^2)\Big]\Big\} \la N,\tX\ra\nonumber\\
&&-{1\over \varrho}\Big\{nH\Big[2\varrho'w^2+\varrho\Big]-\Big(2\la df(\nabla w), \tW\ra+N(\|\tW\|_{\tth}^2)\Big)\Big[2\varrho''w^2-\varrho'\Big]\Big\} \la N,\tX\ra\nonumber\\
&=&-{(\mathrm{log}\chi)'\over
\varrho}\Big[2\varrho'(1-\ms)+\varrho\Big]N(\|\tW\|_{\tth}^2)
 \la N,\tX\ra \nonumber\\
 &&-{1\over \varrho}\Big\{nH\Big[2\varrho'(1-\ms)+\varrho\Big]-2\la df(\nabla w), \tW\ra\Big[2\varrho''(1-\ms)-\varrho'\Big]\Big\}\la N,\tX\ra\nonumber\\
&=&-{(\mathrm{log}\chi)'\over
\varrho}\Phi(\ms)N(\|\tW\|_{\tth}^2)
 \la N,\tX\ra-{1\over \varrho}\Big\{nH\Phi(\ms)-2\la df(\nabla w),
\tW\ra\Phi'(\ms)\Big\}\la N,\tX\ra.\nonumber\eeq  The proof is
completed.\qed

Observing that $\chi=1$ in the BH-case,  we immediately get the following simpler formula, which enables us to study the BH-minimal surface in the hyperbolic Randers space in the next section.

\begin{cor}\label{Cor111} Let  $f:M^n\rightarrow (\tilde{M}^{n+1},\tilde{F})$ be a hypersurface
isometrically immersed in a Randers manifold. Suppose the navigation
data of $(\tilde{M},\tilde{F})$ is $(\tth,\tW)$, where
$\tilde{W}$ is a Killing vector field. Then for any $\tilde{X}\in T\tilde{M}$ along $f$,
 the
BH-mean curvature  of $f$ is given by \[
\mathcal {H}_f(\tilde{X})=-{1\over \varrho}\Big\{nH\Phi(\ms)-2\la
df(\nabla w), \tW\ra_{\tth}\Phi'(\ms)\Big\}\la
N,\tilde{X}\ra_{\tth},\] where the notations are the same as those in Theorem \ref{geomean}  and
$\Phi(\ms)$ is given by (\ref{Phi}) in  the BH-case.
\end{cor}

\begin{rem} If $\tW$ is a Killing field of constant length, then  $\chi$ is a constant both in the BH-case and the HT-case, and the formula in Theorem \ref{geomean}  reduces to the same one as in Corollary \ref{Cor111}, which is Theorem 3.3 in \cite{Cui1}. This enables us to get both nontrivial BH-minimal surfaces and HT-minimal surfaces in the Bao-Shen sphere in \cite{Cui2}. Note that when $\tW$ is Killing only, the formula in Corollary \ref{Cor111} is applicable only for the BH-case.
\end{rem}

\begin{rem} Although Theorem \ref{geomean} only states for the BH-volume form and HT-volume form,  it actually holds for the hypersurface endowed with any volume form  in the form   $dV=\mathcal {F}dx$,  $\mathcal {F}={\varrho\over\chi}\sqrt{det(h_{ij})}$,  where  $\varrho=\varrho(\ms)$ and $\chi=\chi(\|\tW\|_{\tth}^2)$ are arbitrary smooth functions (not only (\ref{volume})).
\end{rem}

\section{ Rotational BH-minimal surfaces of spherical and hyperbolic types}

 Let $(p^1,p^2,p^3,p^4)$ be the
coordinates of $L^4$ with the Lorentzian metric
$g_{L^4}=-(d{p^1})^2+(d{p^2})^2+(d{p^3})^2+(d{p^4})^2,$
and let $H^3\subset L^4$ be the model of upper paraboloid
$-({p^1})^2+({p^2})^2+({p^3})^2+({p^4})^2=-1$ with $p^1>0$. It is well known that the induced metric $\tth$ on $H^3$ has constant sectional curvature $\BK_{\tth}=-1$. For fixed numbers $\e_1\neq 0$ and $\e_2\neq 0$, we consider the Killing  field $\tW=(\e_1p^2,\e_1p^1,-\e_2p^4,\e_2p^3)$ in $H^3$ which is not of constant length and we call it a Killing field of mixed $S$ and $J$ type by using the terminology in \cite{BRS}.
 Note that if $\e_1=\e_2\neq 0$, then $\tW$ is the scaling of the distinguished Killing field $(p^2,p^1,-p^4,p^3)$ in $H^3$.  In order to guarantee $\|\tW\|_{\tth}<1$, we consider the domain
 \be \Omega=H^3\cap\{(p^1,p^2,p^3,p^4)\in L^4; -\e_1^2(p^2)^2+\e_1^2(p^1)^2+\e_2^2(p^4)^2+\e_2^2(p^3)^2<1\}.\label{Omega}\ee
 In this section,  we shall use Corollary \ref{Cor111} to obtain the nontrivial rotational BH-minimal surfaces in the domain $\Omega\subset H^3$ with a Randers metric $\tF$ whose navigation data is $(\tth,\tW)$. By \cite{BRS}, $(\Omega,\tF)$ is a regular Randers space of constant flag curvature $\BK_{\tF}=-1$.

The {\em rotational surface of spherical type} is parametrized by 
 $X: R^2\rightarrow H^3\subset L^4$, \be
X(t,\theta)=(x(t),y(t),z(t)\cos\theta,z(t)\sin\theta),\label{surface1}\ee
where 
\[ x(t):=\sqrt{1+x_1(t)^2}\cosh\phi(t),\ \ \
y(t):=\sqrt{1+x_1(t)^2}\sinh\phi(t),\ \ \  z(t):=x_1(t), \label{phi2}\]
$x_1(t)\neq 0$, and the {\em rotational surface of hyperbolic type} is parametrized by 
 $X: R^2\rightarrow H^3\subset L^4$, \be
X(t,\theta)=(x(t)\cosh\theta,x(t)\sinh\theta,y(t),z(t)),\label{surface2}\ee
where 
\[ x(t):=x_1(t),\ \ \
y(t):=\sqrt{x_1(t)^2-1}\cos\phi(t),\ \ \  z(t):=\sqrt{x_1(t)^2-1}\sin\phi(t), \label{phi2}\]
 $x_1(t)>1$. In both cases, 
$\phi(t):=\int^t_0{\sqrt{\delta+x_1^2-{x}_1'^2}/(\delta+x_1^2)}d\sigma$ where $\delta=1$ for the spherical case and $\delta=-1$ for the hyperbolic case. 
The rotational surface of spherical type is obtained  by rotating the curve $\gamma(t)=(x(t),y(t),z(t),0), t\in R$, in $H^3$ around the
$p^1p^2$-plane,  and the rotational surface of hyperbolic type is obtained  by rotating the curve $\gamma(t)=(x(t),0,y(t),z(t)), t\in R$, in $H^3$ around the
$p^3p^4$-plane, respectively.
One
can check that for both types the parameter $t$ is the parametrization of  arc
length with respect to the standard hyperbolic metric, i.e.,
\[-x'(t)^2+y'(t)^2+z'(t)^2=1.\]

\begin{rem} \label{Special}If $\delta+x_1^2-{x}_1'^2=0$, then the nontrivial solutions are $x_1(t)=\pm \sinh(t+c)$ for $\delta=1$ and $x_1(t)=\pm \cosh(t+c)$ for $\delta=-1$, where $c$ is an arbitrary constant. Then, up to a change of parameters,  (\ref{surface1}) and (\ref{surface2})  reduce to
\be X(t,\theta)=(\cosh t,0,\sinh t \cos\theta,\sinh t \sin\theta)\label{Spe1}\ee and \be X(t,\theta)=(\cosh t\cosh\theta,\cosh t\sinh\theta,  \sinh t,0),\label{Spe2}\ee respectively. They are totally geodesic and so $H=0$, and it follows from the equation (\ref{SecondT}) below that $\la df(\nabla
w), \tW\ra=0$. By  Corollary \ref{Cor111},  they are BH-minimal in $(\Omega,\tF)$.
\end{rem}

 In the following, we shall study the rotational surfaces of  spherical and hyperbolic types simultaneously. We shall consider the minimal surfaces which are not the cases in Remark \ref{Special} (i.e. $\delta+x_1^2-{x}_1'^2\neq 0$).  A unit normal
vector field of $X$ in $H^3$ is, for the spherical case, \be N=\Big(zy'-yz',\ x'z-xz',\
(xy'-x'y)\cos\theta,\ (xy'-x'y)\sin\theta\Big)\label{normal1}\ee
and for the hyperbolic case, \be N=\Big((yz'-zy')\cosh\theta,\ (yz'-zy')\sinh\theta,\ xz'-x'z,\
x'y-xy'\Big),\label{normal2}\ee which can be checked directly by computing that
$\la X_t,N\ra=\la X_\theta,N\ra=\la X,N\ra=0$  and $\la N,N\ra=1$, where the notations $X_t$ and
$X_\theta$ denote the derivatives with respect to $t$ and $\theta$,
and $\la, \ra$ denotes the Lorenz metric of $L^4$. A straightforward
computation shows that the metric matrix $(h_{ij})=(\la \pa_{x^i} X,\pa_{x^j} X\ra)$ of (\ref{surface1}) and (\ref{surface2}) have the uniform form  \begin{equation}\Big[
\begin{array}{cc}
 h_{11} & h_{12}\\
h_{21} & h_{22}
\end{array}\Big]=\Big[
\begin{array}{cc}
 1& 0\\
0 & x_1^2
\end{array}\Big],\nonumber
\end{equation}
where  $x^1:=t$, $x^2:=\theta$.
  The principal
curvatures  are given by $-\lambda_1$ and $-\lambda_2$ with
\[\lambda_1={N_t\over X_t}={x_1-{x}_1''\over\sqrt{\delta+x_1^2-{x}_1'^2}},\ \ \ 
\lambda_2={N_\theta\over X_\theta}={\sqrt{\delta+x_1^2-{x}_1'^2}\over x_1},\]
and therefore the mean curvature of the Riemannian immersion $X: R^2\rightarrow H^3\subset L^4$, corresponding to this normal field $N$,
is given by
\[H=-{\lambda_1+\lambda_2\over 2}={{x_1{x}_1''+{x}_1'^2-2x_1^2-\delta}\over2x_1\sqrt{\delta+x_1^2-{x}_1'^2}},\]
 where $\delta=1$ for the spherical case and $\delta=-1$ for the hyperbolic case. 

The Killing field $\tW$, restricted to the surface (\ref{surface1}) and (\ref{surface2}), is respectively given by 
\[\tW=(\e_1y(t), \e_1x(t), -\e_2z(t)\sin\theta, \e_2z(t)\cos\theta)\]
and  \[ \tW=(\e_1x(t)\sinh\theta, \e_1x(t)\cosh\theta,  -\e_2z(t), \e_2y(t)).\] 
It follows from (\ref{normal1}) and (\ref{normal2}) that, for the spherical type,
\[ w=\la N,\tW \ra=\e_1[-y(y'z-yz')+x(x'z-xz')]=\e_1[(-x^2+y^2)z'+z(xx'-yy')]=-\e_1{x}_1',\]
and for the hyperbolic type,
\[w=\la N,\tW \ra=\e_2[-z(xz'-x'z)+y(x'y-xy')]=\e_2[(y^2+z^2)x'-x(yy'+zz')]=-\e_2{x}_1',\]
which can be written in the uniform form $w=-\e_k{x}_1', k=1,2$, where we are using the identity $-xx'+yy'+zz'=0$, which is derived from $-x^2+y^2+z^2=1$ by taking the derivative with respect to $t$. 

Similarly, by separate computations, we write the following term for the spherical and hyperbolic surfaces in a uniform form \beq\la df(\nabla
w), \tW\ra&=&h^{ij}{\pa w\over \pa
x^j}{\pa
X^A\over\pa x^i}\tW^A=h^{ij}{\pa w\over \pa x^j}\la {\pa
X\over\pa x^i},\tW\ra\nonumber\\
&=&h^{11}w'\la
 {\pa
X\over\pa
t},\tW\ra=-\e_k^2{x}_1''\sqrt{\delta+x_1^2-{x}_1'^2},\label{SecondT}\eeq
where $k=1,2$,  $1\leq i,j\leq2$, $1\leq A\leq4$.
It follows from  Corollary \ref{Cor111} that $X$ is BH-minimal
if and only if
\be{{x_1{x}_1''+{x}_1'^2-2x_1^2-\delta}\over{\delta+x_1^2-{x}_1'^2}}\Phi(\ms)+2\e_k^2x_1{x}_1''\Phi'(\ms)=0,\label{eq}\ee
where $\ms=1-w^2=1-\e_k^2x_1'^2$ and $\Phi(\ms)=(3\ms-2)/\ms^2$.
If $\Phi(\ms)\neq0$, then it is surprising that the first integral of
(\ref{eq}) can be explicitly given by \be
x_1\sqrt{\delta+x_1^2-{x}_1'^2}\Phi(\ms)=E,\label{fir1}\ee where  $E\neq 0$ is a
constant, called the {\em energy} of the minimal surfaces. 

Now we analyze the BH-minimal equation (\ref{eq}).
If $\Phi(\ms)=0$, i.e., $\ms=2/3$, we have $x_1=\pm{1\over \sqrt{3}\e_k}t+c$ where $c$ is a constant. It is obvious that $x_1$ satisfies (\ref{eq}) and then (\ref{surface1}) and  (\ref{surface2}) are really BH-minimal surfaces in the domain $(\Omega,\tF)$.

Next, we consider $\Phi(\ms)\neq0$.    We first claim that $x_1'$ is not a constant  in any neighborhood of $t$. Otherwise, if $x_1'$ is a constant  in a neighborhood of $t$, then $x_1=c_1t+c_2$ with $c_1$ and $c_2$ two constants. Plugging into  (\ref{eq}) yields 
$c_1=0$ and $2c_2^2+\delta=0$, and this implies  $2x_1^2+\delta=0$, which is impossible for (\ref{surface1}) ($\delta=1$) and  (\ref{surface2}) ($\delta=-1$, $x_1^2>1$), contradiction.

 Let $s:={x}_1'^2$. Since $s$ is not a constant, it can be used as a local parameter. 
 Noticing that  $\Phi(\ms)\neq0$ and $\ms=1-\epsilon_k^2
{x}_1'^2=1-\epsilon_k^2 s$, from (\ref{fir1}), we have\be x_1
=\pm{1\over
\sqrt{2}}\Big[-(\delta-s)+\sqrt{(\delta-s)^2+4E^2{(1-\e_k^2s)^4\over (1-3\e_k^2s)^2}}\Big]^{1\over
2}\label{newx212}.\ee

Locally,
\beq\phi&=&\int{d\phi}=\int{d\phi(t)\over dt}{dt\over
dx_1}{dx_1\over ds}ds=\int{d\phi\over dt}{1\over{dx_1\over
dt}}{dx_1\over
ds}ds\nonumber\\
&=&\pm\int{\sqrt{\delta+x_1^2-{x}_1'^2}\over\delta+x_1^2}{1\over
\sqrt{s}}{dx_1\over ds}ds\nonumber\\
&=&\pm{E\over2}\int{1\over x_1^2(\delta+x_1^2)\Phi(\ms)}{1\over
\sqrt{s}}{dx_1^2\over ds}ds\nonumber\\
&=&\pm{E\over2\delta}\int{(1-\e_k^2s)^2\over \sqrt{s}(1-3\e_k^2s)}{d\over
ds}\Big(log{x_1^2\over\delta+x_1^2}\Big)ds,\label{newphi}\eeq where   we are using the equation (\ref{fir1}), $k=1, \delta=1$ for the spherical type and  $k=2, \delta=-1$ for the hyperbolic type.
Therefore, the  BH-minimal surfaces  in the
form (\ref{surface1}) and  (\ref{surface2})  can be respectively reparametrized by
\be
X(s,\theta)=(\sqrt{1+x_1^2}\cosh\phi,\sqrt{1+x_1^2}\sinh\phi,x_1\cos\theta,x_1\sin\theta)\label{surface11}\ee
and
\be
X(s,\theta)=(x_1\cosh\theta,x_1\sinh\theta,\sqrt{x_1^2-1}\cos\phi,\sqrt{x_1^2-1}\sin\phi),\label{surface22}\ee
 where
 $x_1=x_1(s)$ and $\phi=\phi(s)$ are given by (\ref{newx212}) and
 (\ref{newphi}), respectively. Note that in (\ref{newx212}), $s\neq {1/(3\e_k^2)}$. So,
  the parameters are $s\in (0,1/(3\e_k^2))\cup (1/(3\e_k^2),1/\e_k^2)$ and $\theta\in
 S^1$, where $k=1,\delta=1$ for the spherical type and $k=2,\delta=-1$ for the hyperbolic type.

\begin{thm}\label{BHM}  Let $X$ be the rotational surface of spherical type (\ref{surface1}) (resp. the hyperbolic type (\ref{surface2})) in the Randers domain  $\Omega\subset H^3$ defined by (\ref{Omega}) with constant flag curvature $\BK=-1$.  
If $X$ is  BH-minimal in  $(\Omega,\tF)$, then it must be  one of the following surfaces with intersection of $\Omega$:

(a) $X$ is the totally geodesic surface (\ref{Spe1}) (resp. (\ref{Spe2})) in $H^3$, up to a change of parameters;

(b) $X$ is given by  (\ref{surface1}) (resp. (\ref{surface2})) with $x_1=\pm{1\over \sqrt{3}\e_k}t+c$, where $c$ is a constant, $k=1$ (resp. $k=2$);

(c) $X$  can be locally parametrized as  (\ref{surface11}) (resp. (\ref{surface22})) with any energy $E\neq 0$, where  $\delta=1$ (resp. $\delta=-1$),  the parameters are  $s\in (0,1/(3\e_k^2))\cup (1/(3\e_k^2),1/\e_k^2)$, $k=1$ (resp. $k=2$) and $\theta\in
 S^1$.

\end{thm} 

\begin{rem}  If $\e_1\neq0$ and $\e_2=0$, then $\tW=\e_1(\tx^2,\tx^1,0,0)$, one  can also get the same rotational surface of  spherical type in Theorem \ref{BHM}. Similarly,  if $\e_1=0$ and $\e_2\neq 0$, then $\tW=\e_2(0,0,-\tx^4,\tx^3)$, one  gets the same rotational surface of hyperbolic type in Theorem \ref{BHM}. 
\end{rem}

\end{document}